\documentclass[11pt]{amsart}
\usepackage{amssymb}
\usepackage[margin=1in]{geometry}
\usepackage{epsfig}   
\usepackage[foot]{amsaddr}
\usepackage{caption}

\newtheorem{thm}{Theorem}[section]
\newtheorem{prop}[thm]{Proposition}
\newtheorem{lem}[thm]{Lemma}
\newtheorem{cor}[thm]{Corollary}
\newtheorem{conj}[thm]{Conjecture}

\theoremstyle{definition}

\theoremstyle{remark}

\numberwithin{equation}{section}

\newcommand{\Z}{\mathbb{Z}}
\newcommand{\N}{\mathbb{N}}
\newcommand{\X}{X_{\mathbb{Z} / n \mathbb{Z}}}

\begin{document}

\title{Domination and upper domination of direct product graphs}

\author{Colin Defant$^1$}
\address{$^1$Princeton University}
\email{cdefant@princeton.edu}

\author{Sumun Iyer$^2$}
\address{$^2$Williams College}
\email{ssi1@williams.edu}


\begin{abstract}
Let $\X$ denote the unitary Cayley graph of $\Z /n \Z$. We present results on the tightness of the known inequality $\gamma(\X)\leq \gamma_t(\X)\leq g(n)$, where $\gamma$ and $\gamma_t$ denote the domination number and total domination number, respectively, and $g$ is the arithmetic function known as Jacobsthal's function. In particular, we construct integers $n$ with arbitrarily many distinct prime factors such that $\gamma(\X)\leq\gamma_t(\X)\leq g(n)-1$. We give lower bounds for the domination numbers of direct products of complete graphs and present a conjecture for the exact values of the upper domination numbers of direct products of balanced, complete multipartite graphs. 
\end{abstract}

\maketitle

\bigskip





\section{Introduction}
If $R$ is a commutative ring with unity, we can define the unitary Cayley graph of $R$, denoted $X_R$, as follows. The vertices of $X_R$ are the elements of $R$ and $x$ is adjacent to $y$ if and only if $x-y$ is a unit of $R$. In this paper we study the domination number and upper domination number of $\X$. Motivation for studying $\X$ comes from the theory of graph representation. See Gallian's ``Dynamic Survey of Graph Labeling" for more information about the representation numbers of graphs and for additional references \cite{Gallian}. The unitary Cayley graph of $\Z / n\Z$ is highly symmetric and structured, and graph invariants of $\X$ are well-studied. Often the innate structure of $\X$ gives rise to pleasing combinatorial results. In 1995 Dejter and Giudici \cite{dejter} introduced the notion of a unitary Cayley graph and determined the number of triangles in $\X$. One of the current authors later generalized this result by finding a formula for the number of cliques of any order in $\X$ \cite{colin}. In 2007 Klotz and Sander determined the chromatic number, clique number, independence number, and diameter of $\X$ \cite{klotzsander}. Other properties of unitary Cayley graphs are studied in \cite{coolnames, colin, fuchs, mandmclique}. 



It is natural to view unitary Cayley graphs as direct products of balanced, complete multipartite graphs. Throughout this paper let $V(G)$ denote the vertex set of a graph $G$. If $G$ and $H$ are graphs, then the \emph{direct product} (alternatively called the \emph{tensor product} or \emph{Kronecker product}) of $G$ and $H$, denoted $G\times H$ (some authors use $G\otimes H$), is defined as follows: $V(G\times H)$ is the Cartesian product $V(G) \times V(H)$, and $(g_1, h_1)$ is adjacent to $(g_2,h_2)$ if and only if $g_1$ is adjacent to $g_2$ in $G$ and $h_1$ is adjacent to $h_2$ in $H$. A \emph{balanced, complete $k$-partite graph} is a graph whose vertices can be partitioned into $k$ different independent sets of equal cardinality such that any two vertices in different independent sets are adjacent. The equal-sized independent sets are called the \emph{partite sets}. We denote by $K[a,b]$ the balanced, complete $b$-partite graph in which each partite set has size $a$. Note that $K[1,b]$ is simply the complete graph $K_b$. 

If $p$ is a prime and $\alpha$ is a positive integer, then it is straightforward to see that $X_{\Z/p^\alpha\Z}\cong K[p^{\alpha-1},p]$. It follows from the Chinese remainder theorem that if $n=p_1^{\alpha_1} \cdots p_k^{\alpha_k}$ is the prime factorization of an integer $n>1$, then $\X \cong K[p_1^{\alpha_1-1},p_1] \times \cdots \times K[p_k^{\alpha_k-1},p_k]$. The authors of \cite{coolnames} have shown more generally that the unitary Cayley graph of any finite commutative ring is a direct product of balanced, complete multipartite graphs. Therefore, we will state many of our results in the more general framework of direct products of balanced, complete multipartite graphs. 

This article focuses primarily on two well-studied graph parameters related to dominating sets. We say a vertex $u$ of a graph $G$ \emph{dominates} a vertex $v$ if $u=v$ or $u$ is adjacent to $v$. A \emph{dominating set} of $G$ is a set $D\subseteq V(G)$ such that every vertex in $V(G)$ is dominated by an element of $D$. The \emph{domination number} of $G$, denoted $\gamma(G)$, is the minimum cardinality of a dominating set of $G$. We call a dominating set $D$ \emph{minimal} if no proper subset of $D$ is a dominating set. The \emph{upper domination number} of $G$, denoted $\Gamma(G)$, is the maximum size of a minimal dominating set of $G$. We also find it convenient to define a \emph{total dominating set} of $G$ to be a set $D\subseteq V(G)$ such that every vertex in $V(G)$ is adjacent to an element of $D$. The minimum cardinality of a total dominating set of $G$, called the \emph{total domination number} of $G$, is denoted by $\gamma_t(G)$. Since every total dominating set is a dominating set, we have the trivial inequality $\gamma_t(G)\geq\gamma(G)$. For much more information about domination in graphs, especially in graph products, see \cite{bresar, hede, mekis, Nowakowski} and the references therein.   

In 2010 Meki{\u s} provided bounds for the domination numbers of certain direct products of complete graphs. We restate some of these results in Theorem \ref{Thm1} and devote the rest of that section to developing techniques for proving further bounds. For example, one specific application of our results shows that if $2=n_1\leq n_2\leq n_3\leq n_4$ and $G=\prod_{i=1}^4K_{n_i}$, then $\gamma(G)=8$ (the product denotes the graph direct product).

Let $g(n)$ denote the smallest positive integer $m$ such that every set of $m$ consecutive integers contains an element that is relatively prime to $n$. The arithmetic function $g$ is known as Jacobsthal's 
function; it has received a fair amount of attention from number theorists, partly because of its applications to the study of prime gaps and the study of the smallest primes in arithmetic progressions \cite{Iwaniec, Maier, Pintz, Pomerance}.
In 2013 Maheswari and Manjuri \cite{mandm} claimed that $\gamma(\X)=g(n)$ when $n$ has at least $3$ distinct prime factors. Their proof correctly shows that $\gamma(\X)\leq g(n)$. This is simply because $\{0,1,\ldots,g(n)-1\}$ is a dominating set of $\X$. In fact, this set is a total dominating set of $\X$, so we actually know the stronger inequality $\gamma_t(\X)\leq g(n)$. However, in 2016 one of the current authors \cite{colin} noted that $\gamma(X_{\Z/30\Z})=4<6=g(30)$. In general, $\gamma(\X)$ is not necessarily equal to $g(n)$. In Section 3 we provide results that help to quantify when and how drastically the inequality $\gamma(\X)\leq g(n)$ fails to be an equality. Specifically, we show that for each positive integer $j$, there is an integer $n$ with more than $j$ distinct prime factors such that $\gamma(\X)\leq\gamma_t(\X)<g(n)$. 
  
In Section 4 we conjecture that $\Gamma(\X)=n/p_1$, where $p_1$ is the smallest prime factor of $n$. We prove this conjecture for all $n$ where $p_1=2$ and in some additional cases. We state the conjecture and our results in the more general setting of direct products of balanced, complete multipartite graphs.

\section{Domination in Direct Products of Complete Graphs}

In this section, we develop techniques for proving estimates for the domination numbers of direct products of complete graphs that are independent of our focus on unitary Cayley graphs. We generalize a theorem of Meki{\v s} in Theorem \ref{Thm2}. The only result from this section that will be invoked in subsequent sections is Theorem \ref{cubecorner}, which states that $\gamma(G)=8$ when $G$ is the direct product of $K_2$ and three other complete graphs. Therefore, the reader interested only in the subsequent sections may safely pass over the current one. 

In \cite{mekis}, Meki{\v s} studied the domination numbers of graphs of the form $\prod_{i=1}^tK_{n_i}$, where $K_n$ denotes the complete graph on $n$ vertices (recall that the product denotes the graph direct product). For completeness, we summarize some of his results in the following theorem. 

\begin{thm}[Meki{\v s}]\label{Thm1}
Let $G=\prod_{i=1}^tK_{n_i}$, where $2\leq n_1\leq n_2\leq\cdots\leq n_t$. If $t=2$, then \[\gamma(G)=\begin{cases} 2, & \mbox{if } n_1=2; \\ 3, & \mbox{if } n_1\geq 3. \end{cases}\] If $t=3$, then $\gamma(G)=4$. For $t\geq 3$, we have $\gamma(G)\geq t+1$, and equality holds if $n_1\geq t+1$. 
\end{thm} 

Even when considering the domination numbers of more general direct products of balanced, complete multipartite graphs, it is useful to know lower bounds for the domination numbers of direct products of complete graphs. This is because of the following lemma, whose straightforward proof we omit. 

\begin{lem}\label{Lem5}
For any positive integers $a_1,a_2,\ldots,a_t,b_1,b_2,\ldots,b_t$, we have \[\gamma\left(\prod_{i=1}^tK[a_i,b_i]\right)\geq\gamma\left(\prod_{i=1}^tK_{b_i}\right).\]
\end{lem}
The next lemma builds upon the last line in Theorem \ref{Thm1} by giving upper bounds for $\gamma_t(G)$ (hence, also for $\gamma(G)$) under specific conditions on the sizes of $n_1$ and $n_2$. Recall that the vertices of the graph $\prod_{i=1}^tK_{n_i}$ are $t$-tuples in which the $i^{\text th}$ coordinate is a vertex in $K_{n_i}$. Throughout the rest of this section, we denote the $i^\text{th}$ coordinate of a vertex $x$ in this direct product by $[x]_i$. Vertices $x$ and $y$ are adjacent if and only if $[x]_i\neq [y]_i$ for all $1\leq i\leq t$. 

\begin{lem}\label{Lem3}
Let $G=\prod_{i=1}^tK_{n_i}$, where $2\leq n_1\leq n_2\leq\cdots\leq n_t$ and $t\geq 3$. If $m$ is a nonnegative integer such that $\dfrac{t+m}{m+1}<n_1$ and $t+m<n_2$, then $\gamma(G)\leq\gamma_t(G)\leq t+m+1$.  
\end{lem}

\begin{proof}
It is convenient to think of the vertices of $K_{n_i}$ as the elements of $\Z/n_i\Z$ (although we still think of the vertex sets of $K_{n_i}$ and $K_{n_j}$ as disjoint when $i\neq j$). Let $y_r$ be the vertex $(r,r,\ldots,r)$ of $G$, where the $i^\text{th}$ coordinate is taken modulo $n_i$. We claim that every vertex of $G$ is adjacent to an element of the set $D=\{y_0,y_1,\ldots, y_{t+m}\}$. In other words, $D$ is a total dominating set for $G$. To see this, suppose instead that there is a vertex $a\in V(G)$ that is not adjacent to any element of $D$. For each $\ell\in\{0,1,\ldots,t+m\}$, there is an index $\beta(\ell)\in\{1,\ldots,t\}$ such that $[a]_{\beta(\ell)}=[y_\ell]_{\beta(\ell)}=\ell\pmod{n_{\beta(\ell)}}$. Assume that $\beta(\ell)=\beta(\ell')\in\{2,3,\ldots,t\}$ for some $\ell,\ell'\in\{0,1,\ldots,t+m\}$. This implies that $\ell\pmod{n_{\beta(\ell)}}=[a]_{\beta(\ell)}=[a]_{\beta(\ell')}=\ell'\pmod{n_{\beta(\ell)}}$. Since $|\ell-\ell'|\leq t+m<n_2\leq n_{\beta(\ell)}$, we must have $\ell=\ell'$. This shows that for each $s\in\{2,3,\ldots,t\}$, $|\beta^{-1}(s)|\leq 1$. Hence, $|\beta^{-1}(1)|\geq(t+m+1)-(t-1)=m+2$.  Choose $\ell_0,\ell_1,\ldots,\ell_{m+1}\in\beta^{-1}(1)$ with $\ell_0<\ell_1<\cdots<\ell_{m+1}$. Since $[a]_1=\ell_i\pmod{n_1}$ for all $0\leq i\leq m+1$, it follows that $\ell_{m+1}-\ell_0\geq(m+1)n_1$. We also know that $\ell_{m+1}-\ell_0\leq t+m$ since $\ell_0,\ell_{m+1}\in\{0,1,\ldots,t+m\}$. This shows that $(m+1)n_1\leq t+m$, contradicting the hypothesis.        
\end{proof}

The purpose of the rest of this section is to extend the above theorem of Meki{\v s} by proving additional lower bounds for the domination numbers of direct products of complete graphs. The last statement in Theorem \ref{Thm1} tells us that the difficulty in calculating these domination numbers arises when some of the sizes of the complete graphs (the numbers $n_i$) are small relative to $t$, the total number of terms in the direct product. Therefore, it will prove useful to first reduce to the case in which at most one of the complete graphs in our direct product is $K_2$. 

Recall that the disjoint union of two graphs $G_1$ and $G_2$, denoted $G_1\oplus G_2$, is the graph whose vertex set is the disjoint union of the vertex sets of $G_1$ and $G_2$ and whose edge set is the disjoint union of the edge sets of $G_1$ and $G_2$. In other words, $G_1\oplus G_2$ is formed by taking one copy of $G_1$ and one (disjoint) copy of $G_2$. It is well-known \cite{Campanelli} that the disjoint union and direct product satisfy the distributive law $G_1\times (G_2\oplus G_3)\cong (G_1\times G_2)\oplus (G_1\times G_3)$.  


It is straightforward to show that $\prod_{i=1}^s K_2\cong\bigoplus_{i=1}^{2^{s-1}}K_2$. For example, $K_2\times K_2\times K_2$ is isomorphic to the disjoint union of $4$ copies of $K_2$. By the above distributive law, we see that for any graph $H$, 
\begin{equation}\label{Eq1}
\left(\prod_{i=1}^sK_2\right)\times H\cong\bigoplus_{i=1}^{2^{s-1}}(K_2\times H).
\end{equation} 
The following lemma now follows as a simple corollary to \eqref{Eq1}. 

\begin{lem}\label{Lem1}
Let $G=\left(\prod_{i=1}^sK_2\right)\times H$, where $s$ is a positive integer and $H$ is a finite simple graph. We have \[\gamma(G)=2^{s-1}\gamma(K_2\times H).\] 
\end{lem}

Throughout the remainder of this section, we estimate the domination numbers of graphs of the form $\prod_{i=1}^tK_{n_i}$, where $2\leq n_1\leq n_2\leq\cdots\leq n_t$. In doing so, we may assume (because of Theorem \ref{Thm1}) that $t\geq 4$. Because of the preceding lemma, we may also assume $n_2\geq 3$. The following seemingly technical lemma provides a very useful technique for gaining information about minimum dominating sets in the graphs we are considering.  

\begin{lem}\label{Lem2}
Let $G=\prod_{i=1}^tK_{n_i}$, where $2\leq n_1\leq n_2\leq\cdots\leq n_t$, $t\geq 4$, and $n_2\geq 3$. Let $D$ be a dominating set of $G$ of minimum size. Let $E_1,\ldots,E_k$ be nonempty disjoint subsets of $D$ such that $|E|\geq\gamma(G)-t+k+1$, where $E=\bigcup_{j=1}^kE_j$. Suppose that there exist distinct integers $i_1,\ldots,i_k\in\{1,\ldots,t\}$ such that for each $j\in\{1,\ldots,k\}$, all elements of $E_j$ have the same $i_j^\text{th}$ coordinate. Let $h=\max(\{1,\ldots,t\}\setminus\{i_1,\ldots,i_k\})$. Then $n_h\in\{2,3\}$ and $|E|\leq\gamma(G)-t+k+2$. If $|E|=\gamma(G)-t+k+2$, then $\{1,\ldots,t\}\setminus\{i_1,\ldots,i_k\}=\{1,h\}$, $n_1=2$, $n_h=3$, and $E=D$.   
\end{lem}

The hypothesis that $|E|\geq \gamma(G)-t+k+1$
 can be rewritten as \[\sum_{j=1}^k(|E_j|-1)\geq |D|-t+1.\] Roughly speaking, the lemma says that if we can construct disjoint subsets $E_1,\ldots,E_k$ of $D$ and distinct integers $i_1,\ldots,i_k\in\{1,\ldots,t\}$ so that all the vertices in $E_j$ agree in their $i_j^\text{th}$ coordinates, then $|D|=\gamma(G)$ cannot be too small relative to the size of the union $E=\bigcup_{j=1}^kE_j$. The last sentence in the lemma states that if $|D|$ happens to be small enough so that $|E|=|D|-t+k+2$, then we can obtain very precise restrictions on the values of $h$, $n_1$, and $n_h$. In practice, these restrictions can be used to obtain a contradiction and eliminate this case completely (for example, this case is impossible if we assume that $n_1\geq 3$ or $n_2\geq 4$). 
 
\begin{proof}
Because $\gamma(G)=|D|\geq |E|\geq\gamma(G)-t+k+1$, we must have $k\leq t-1$. This guarantees that $h$ actually exists. Let $a=\vert D\setminus E\vert=\gamma(G)-\vert E\vert$, and note that $a\leq t-k-1$ by hypothesis. Let $\{\ell_1,\ldots,\ell_{t-k}\}=\{1,\ldots,t\}\setminus\{i_1,\ldots,i_k\}$, where $\ell_1<\cdots<\ell_{t-k}$. Then $h=\ell_{t-k}$. Let $D\setminus E=\{d_{\ell_1},\ldots,d_{\ell_a}\}$. For each $q\in\{a+1,\ldots,t-k-1\}$ (this set might be empty), let $d_{\ell_q}$ be an arbitrary vertex of $G$. For each $j\in\{1,\ldots,k\}$, let $d_{i_j}$ be an element of $E_j$. We have chosen a vertex $d_r$ for each $r\in(\{1,\ldots,t\}\setminus \{h\})$. Observe that there are exactly $n_h$ vertices $x\in V(G)$ that satisfy 
\begin{equation}\label{Eq2}
[x]_r=[d_r]_r\text{ for all }r\in(\{1,\ldots,t\}\setminus \{h\}).
\end{equation}
Fix such a vertex $x$. If $b\in E_j$ for some $j\in\{1,\ldots,k\}$, then $x$ is not adjacent to $b$ because $[x]_{i_j}=[d_{i_j}]_{i_j}=[b]_{i_j}$. We also know that $x$ is not adjacent to any element $d_{\ell_s}$ of $D\setminus E$ because $[x]_{\ell_s}=[d_{\ell_s}]_{\ell_s}$. This shows that $x$ is not adjacent to any element of $D$. Since $D$ dominates $G$, $x\in D$. 

Suppose $n_h\geq 4$, and let $x_1,x_2,x_3,x_4$ be four distinct vertices of $G$ that satisfy \eqref{Eq1}. We have shown that all four of these vertices are elements of the dominating set $D$. Moreover, $[x_1]_r=[x_2]_r=[x_3]_r=[x_4]_r$ for all $r\in\{1,\ldots,t\}\setminus\{h\}$. It follows that any vertex in $G$ that is adjacent to $x_3$ or $x_4$ is also adjacent to $x_1$ or $x_2$. Let $y$ be a vertex of $G$ that is adjacent to both $x_3$ and $x_4$. Then $(D\cup\{y\})\setminus\{x_3,x_4\}$ dominates $G$, contradicting the assumption that $|D|=\gamma(G)$ is the smallest size of a dominating set of $G$. We conclude that $n_{h}\in\{2,3\}$.  

To prove the rest of the lemma, assume that $|E|\geq\gamma(G)-t+k+2$. Equivalently, $a\leq t-k-2$. This implies that there is a second-largest element $h'=\ell_{t-k-1}$ of $\{1,\ldots,t\}\setminus\{i_1,\ldots,i_k\}$. In symbols, \[h'=\max(\{1,\ldots,t\}\setminus(\{i_1,\ldots,i_k\}\cup\{h\})).\] We wish to show that $n_{h'}=2$. Since $n_h\in\{2,3\}$ and $n_2\geq 3$, this will show that $\{1,\ldots,t\}\setminus\{i_1,\ldots,i_k\}=\{1,h\}$, $n_1=2$, and $n_h=3$. This, in turn, will mean that $k=t-2$ so that $|E|=\gamma(G)-t+k+2=|D|$. Of course, this will imply that $E=D$.  

Assume by way of contradiction that $n_{h'}\geq 3$. Let $Z$ be the set of all vertices $x\in V(G)$ that satisfy 
\begin{equation}\label{Eq3}
[x]_r=[d_r]_r\text{ for all }r\in(\{1,\ldots,t\}\setminus \{h',h\}).
\end{equation}
Note that $|Z|=n_{h'}n_h\geq 9$. By the same argument used before, we find that $Z\subseteq D$. Choose distinct $z_1,z_2\in Z$. Because $n_h\geq n_{h'}\geq 3$, there exists a vertex $y$ of $G$ that is adjacent to both $z_1$ and $z_2$. If a vertex $v$ is adjacent to either $z_1$ or $z_2$, then it must be adjacent to some vertex in $Z\setminus\{z_1,z_2\}$. This implies that $(D\cup\{y\})\setminus\{z_1,z_2\}$ is a dominating set of $G$. As before, this contradicts the fact that $|D|=\gamma(G)$.  
\end{proof}

It would be interesting to try strengthening the preceding lemma; doing so could lead to stronger versions of the results below or shorter proofs thereof. For example, it might be possible to show that $|E|\leq\gamma(G)-t+k+1$ in all cases so that the last sentence of the lemma is vacuously true. We illustrate the utility of Lemma \ref{Lem2} in proving lower bounds for $\gamma(G)$ in the proof of the following theorem. 

\begin{thm}\label{Thm2}
Let $G=\prod_{i=1}^tK_{n_i}$, where $2\leq n_1\leq n_2\leq\cdots\leq n_t$, $t\geq 4$, and $n_2\geq 3$. We have \[\gamma(G)\geq t+1+\left\lfloor\frac{t-1}{n_1-1}\right\rfloor\]
\end{thm}
\begin{proof}
Put $m=\left\lfloor\dfrac{t-1}{n_1-1}\right\rfloor-1$. A simple manipulation shows that $n_1\leq\dfrac{t+m}{m+1}$. Suppose by way of contradiction that $\gamma(G)\leq t+m+1$. Note that $\gamma(G)\geq t+1$ by Theorem \ref{Thm1}. Writing $\gamma(G)=t+m'+1$, where $m'\leq m$, we find that \[n_1\leq\frac{t+m}{m+1}\leq\frac{t+m'}{m'+1}=\frac{\gamma(G)-1}{\gamma(G)-t}.\] 
Thus, $\gamma(G)-t\leq\left\lfloor\dfrac{\gamma(G)-1}{n_1}\right\rfloor$. For each vertex $v$ of $K_{n_1}$, let $F_v$ be the set of vertices in $D$ with first coordinate $v$. By the pigeonhole principle, there exists a vertex $v^*$ of $K_{n_1}$ such that $\vert F_{v^*}\vert\geq\left\lfloor\dfrac{\gamma(G)-1}{n_1}\right\rfloor+1$. If we set $k=1$, $E_1=F_{v^*}$, and $i_1=1$ in Lemma \ref{Lem2}, then we find that $|F_{v^*}|\leq\gamma(G)-t+3$, where equality can only hold if $\{1,\ldots,t\}\setminus\{i_1\}=\{1,h\}$. Since $t\geq 4$, equality cannot hold. Thus, \begin{equation}\label{Eq4}
\gamma(G)-t+1\leq|F_{v^*}|\leq\gamma(G)-t+2.
\end{equation} 
The inequality $2\leq n_1\leq\dfrac{t+m}{m+1}$ forces $m\leq t-2$. Using \eqref{Eq4} and the assumption that $\gamma(G)\leq t+m+1$, we find that
\begin{equation}\label{Eq5}
|F_{v^*}|\leq\gamma(G)-t+2\leq(t+m+1)-t+2=m+3\leq t+1.
\end{equation}
 
Let $D=\{d_1,\ldots,d_{\gamma(G)}\}$, where $\{d_1,d_{t+1},d_{t+2},\ldots,d_{\gamma(G)}\}\subseteq F_{v^*}$. Let $S_t$ denote the symmetric group on $t$ letters, and let $S_t(1)=\{\sigma\in S_t\colon\sigma(1)=1\}$. For each $\sigma\in S_t(1)$, there is a vertex $x_\sigma\in V(G)$ with $[x_\sigma]_i=[d_{\sigma(i)}]_i$ for all $i\in\{1,\ldots,t\}$. By construction, any such $x_\sigma$ must be an element of $D$ because it is not adjacent to any elements of $D$. Furthermore, any such $x_\sigma$ must be in $F_{v^*}$ since its first coordinate is the same as that of $d_1$. Thus, we have a map $f\colon S_t(1)\to F_{v^*}$ given by $f(\sigma)=x_\sigma$. 

Assume for the moment that $t\geq 5$. Using \eqref{Eq5}, we find that \[\frac{|S_t(1)|}{|F_{v^*}|}\geq\frac{(t-1)!}{t+1}\geq 4.\] This implies that $|f^{-1}(z)|\geq 4$ for some $z\in F_{v^*}$. Choose distinct $\sigma_1,\sigma_2,\sigma_3,\sigma_4\in f^{-1}(z)$. It is straightforward to show that one of the following must hold (possibly after reindexing $\sigma_1,\ldots,\sigma_4$): 

\begin{enumerate}
\item There is some $q\in\{2,3,\ldots,t\}$ such that $\sigma_1(q),\sigma_2(q),\sigma_3(q),\sigma_4(q)$ are all distinct.
\item There are distinct $q,q'\in\{2,3,\ldots,t\}$ such that $\sigma_1(q)\neq\sigma_2(q)$ and $\sigma_3(q')\neq \sigma_4(q')$. 
\end{enumerate}

Suppose (1) holds. We have $[z]_q=[x_{\sigma_i}]_q=[d_{\sigma_i(q)}]_q$ for all $1\leq i\leq 4$. Setting $k=2$, $E_1=\{d_1,d_{t+1},d_{t+2},\ldots,d_{\gamma(G)}\}$, $E_2=\{d_{\sigma_1(q)},
d_{\sigma_2(q)},d_{\sigma_3(q)},
d_{\sigma_4(q)}\}$, $i_1=1$, and $i_2=q$ in Lemma \ref{Lem2} shows that $|E_1\cup E_2|\leq\gamma(G)-t+4$, which is a contradiction. Therefore, (2) must hold. We now put $k=3$, $E_1=\{d_1,d_{t+1},d_{t+2},\ldots,d_{\gamma(G)}\}$, $E_2=\{d_{\sigma_1(q)},
d_{\sigma_2(q)}\}$, $E_3=\{d_{\sigma_3(q')},
d_{\sigma_4(q')}\}$, $i_1=1$, $i_2=q$, and $i_3=q'$ in Lemma \ref{Lem2}. In this case, $|E_1\cup E_2\cup E_3|=\gamma(G)-t+5=\gamma(G)-t+k+2$. The last line in the lemma tells us that $\{1,\ldots,t\}\setminus\{i_1,i_2,i_3\}=\{1,h\}$, which is a contradiction because $i_1=1$. From this contradiction, we deduce that $t=4$. We know from \eqref{Eq4} that $\gamma(G)-3\leq |F_{v^*}|\leq\gamma(G)-2$. We consider two cases. 

{\bf Case 1.} $|F_{v^*}|=\gamma(G)-3$. \\ We saw above that $m\leq t-2=2$, and we are assuming that $\gamma(G)\leq t+m+1=m+5$. Consequently, $|F_{v^*}|\leq 4$. With notation as above, $F_{v^*}=D\setminus\{d_2,d_3,d_4\}$. The map $f\colon S_4(1)\to F_{v^*}$ from above is not injective since $|S_4(1)|=6>|F_{v^*}|$. In other words, there exist distinct $\tau,\tau'\in S_4(1)$ such that $x_\tau=x_{\tau'}$. Since $\tau\neq\tau'$, there is some $j\in\{2,3,4\}$ such that $\tau(j)\neq\tau'(j)$. We have $[d_{\tau(j)}]_j=[x_\tau]_j=[x_{\tau'}]_j=[d_{\tau'(j)}]_j$. Let $\theta$ be the element of $\{2,3,4\}$ that is not $\tau(j)$ or $\tau'(j)$. Let $\{2,3,4\}\setminus\{j\}=\{j',j''\}$. Let $A$ be the set of all vertices of $G$ with first coordinate $v^*$ and $j^\text{th}$ coordinate equal to $[d_{\tau(j)}]_j$. Since $3\leq n_2\leq n_3\leq n_4$, there are at least $5$ vertices $z\in A$ satisfying either $[z]_{j'}=[d_\theta]_{j'}$ or $[z]_{j''}=[d_\theta]_{j''}$. By construction, none of these $5$ vertices can be adjacent to any of the elements of $D$. Since $D$ is a dominating set of $G$, every one of these $5$ vertices must be in $D$. It follows that these $5$ vertices are all in $F_{v^*}$, which contradicts the fact that $|F_{v^*}|\leq 4$.  

{\bf Case 2.} $|F_{v^*}|=\gamma(G)-2$. \\ 
Say $D\setminus F_{v^*}=\{y,y'\}$. Let $y=(p,q,r,s)$ and $y'=(p',q',r',s')$. Consider the set $B$ of all vertices of $G$ with first coordinate $v^*$ and second coordinate $q$. Note that $|B|\geq 9$ since $3\leq n_3\leq n_4$. If $q=q'$, then no element of $B$ is adjacent to any element of $D$. Since $D$ dominates $G$, we must have $B\subseteq F_{v^*}$ if $q=q'$. Because $|F_{v^*}|\leq 5$ by \eqref{Eq5}, it follows that $q\neq q'$. Let $u,u',u''$ be three distinct vertices of $K_{n_4}$. None of the vertices \[(v^*,q,r',u),\hspace{.2cm}(v^*,q,r',u'),\hspace{.2cm}(v^*,q,r',u''),\hspace{.2cm}(v^*,q',r,u),\hspace{.2cm}(v^*,q',r,u'),\hspace{.2cm}(v^*,q',r,u'')\] are adjacent to any elements of $D$, so they must all be elements of $F_{v^*}$. Again, this contradicts the fact that $|F_{v^*}|\leq 5$. This is our final contradiction, so the proof is complete. 
\end{proof}

Let $G$ be as in the preceding theorem. The last statement in Theorem \ref{Thm1} tells us that $\gamma(G)\geq t+1$, where equality holds if $n_1\geq t+1$. Theorem \ref{Thm2} yields a converse to this statement. Namely, if $n_1\leq t$, then $\gamma(G)\geq t+2$. Under the slightly stronger additional assumption that $n_3\geq t+1$, Theorem \ref{Thm3} below characterizes when $\gamma(G)=t+2$. First, we prove the following lemma. 

\begin{lem}\label{Lem4}
Let $G=\prod_{i=1}^tK_{n_i}$, where $2\leq n_1\leq n_2\leq\cdots\leq n_t$, $t\geq 4$, and $n_2\geq 3$. Suppose $\gamma(G)=t+2$, and let $D$ be a dominating set of $G$ with $|D|=t+2$.  For every $\ell\in\{1,\ldots,t\}$ and every vertex $v$ of $K_{n_\ell}$, let $F_v(\ell)=\{z\in D\colon[z]_\ell=v\}$. We have $|F_v(\ell)|\leq 2$ for every choice of $\ell$ and $v$. If $\ell,\ell'\in\{1,\ldots,t\}$ are distinct and there exist vertices $v$ of $K_{n_\ell}$ and $v'$ of $K_{n_{\ell'}}$ such that $|F_v(\ell)|=|F_{v'}(\ell')|=2$, then $F_v(\ell)\cap F_{v'}(\ell')\neq\emptyset$.  
\end{lem}

\begin{proof}
Suppose instead that $|F_v(\ell)|\geq 3$. Write $D=\{d_1,d_2,\ldots,d_{t+2}\}$, where $d_\ell,d_{t+1},d_{t+2}\in F_v(\ell)$. If we put $k=1$, $E_1=F_v(\ell)$, and $i_1=\ell$ in Lemma \ref{Lem2}, then we find that 
\begin{equation}\label{Eq6}
|F_v(\ell)|\leq\gamma(G)-t+k+1=4.
\end{equation} It follows from Theorem \ref{Thm2} that $\dfrac{t+1}{2}<n_1$. 

Let $S_t$ be the symmetric group on $t$ letters, and let $S_t(\ell)=\{\sigma\in S_t\colon\sigma(\ell)=\ell\}$. For each $\sigma\in S_t(\ell)$, let $x_\sigma$ be the vertex of $G$ satisfying $[x_\sigma]_i=[d_{\sigma(i)}]_i$ for all $1\leq i\leq t$. For each $\sigma\in S_t(\ell)$, $x_\sigma$ is not adjacent to any element of $D$. It follows that each vertex $x_\sigma$ is in $F_v(\ell)$. Since $|S_t(\ell)|\geq 6>|F_v(\ell)|$, there are distinct $\sigma,\sigma'\in S_t(\ell)$ such that $x_\sigma=x_{\sigma'}$. There exists $\alpha\in\{1,\ldots,t\}\setminus\{\ell\}$ with $\sigma(\alpha)\neq\sigma'(\alpha)$. We have $[d_{\sigma(\alpha)}]_\alpha=[x_\sigma]_\alpha=[x_{\sigma'}]_\alpha=[d_{\sigma'(\alpha)}]_\alpha$. Putting $k=2$, $E_1=\{d_\ell,d_{t+1},d_{t+2}\}$, $E_2=\{d_{\sigma(\alpha)},d_{\sigma'(\alpha)}\}$, $i_1=\ell$, and $i_2=\alpha$ in Lemma \ref{Lem2} tells us that $n_h\in\{2,3\}$, where $h=\max(\{1,\ldots,t\}\setminus\{\ell,\alpha\})$. In particular, $\dfrac{t+1}{2}<n_1\leq n_h\leq 3$. This forces $t=4$, so $|D|=\gamma(G)=6$. 

Let $\{1,2,3,4\}\setminus\{\ell,\alpha\}=\{\theta_1,\theta_2\}$, and let $\{c\}=\{d_1,d_2,d_3,d_4\}\setminus\{d_\ell,d_{\sigma(\alpha)},d_{\sigma'(\alpha)}\}$. Since $n_1>\dfrac{t+1}{2}>2$, there are at least $5$ vertices $y$ of $G$ that satisfy \[[y]_\ell=[d_\ell]_\ell=[d_5]_\ell=[d_6]_\ell,\hspace{.4cm}[y]_\alpha=[d_{\sigma(\alpha)}]_\alpha=[d_{\sigma'(\alpha)}]_\alpha,\hspace{.4cm}\text{and either}\hspace{.4cm}[y]_\theta=[c]_\theta\hspace{.4cm}\text{or}\hspace{.4cm}[y]_{\theta'}=[c]_{\theta'}.\] If $y$ is a vertex with coordinates satisfying these conditions, then $y\in D$ because $y$ is not adjacent to any element of $D$. Since $[y]_\ell=[d_\ell]_\ell=v$, $y\in F_v(\ell)$. This shows that $5\leq |F_v(\ell)|$, which contradicts \eqref{Eq6}. Consequently, $|F_v(\ell)|\leq 2$ for all choices of $\ell$ and $v$. 

Next, suppose $\ell,\ell'\in\{1,\ldots,t\}$ are distinct and that there are vertices $v$ of $K_{n_\ell}$ and $v'$ of $K_{n_{\ell'}}$ such that $|F_v(\ell)|=|F_{v'}(\ell')|=2$. Suppose $F_v(\ell)\cap F_{v'}(\ell')=\emptyset$, and write $D=\{a_1,a_2,\ldots,a_{t+2}\}$, where $F_v(\ell)=\{a_\ell,a_{t+1}\}$ and $F_{v'}(\ell')=\{a_{\ell'},a_{t+2}\}$. Let $z$ be the vertex of $G$ that satisfies $[z]_i=[a_i]_i$ for all $i\in\{1,\ldots,t\}$. Observe that $z$ is not adjacent to any elements of $D$. Consequently, $z\in D$. Since $[z]_\ell=[a_\ell]_\ell=[a_{t+1}]_\ell=v$ and $[z]_{\ell'}=[a_{\ell'}]_{\ell'}=[a_{t+2}]_{\ell'}=v'$, we see that $z\in F_v(\ell)\cap F_{v'}(\ell')$.  
\end{proof}

\begin{thm}\label{Thm3}
Let $G=\prod_{i=1}^tK_{n_i}$, where $2\leq n_1\leq n_2\leq\cdots\leq n_t$, $t\geq 4$, and $n_2\geq 3$. If $n_3\geq t+1$, then $\gamma(G)=t+2$ if and only if one of the following holds: 
\begin{enumerate}
\item $n_1=t$
\item $\dfrac{t+1}{2}<n_1\leq t-1$ and $t+1<n_2$.
\end{enumerate}
\end{thm}

\begin{proof}
If $\displaystyle \frac{t+1}{2}<n_1\leq t-1$ and $t+1<n_2$, then we may invoke Theorem \ref{Thm2} to find that $\gamma(G)\geq t+2$. Putting $m=1$ in Lemma \ref{Lem3} then shows that $\gamma(G)=t+2$. 

Next, assume $n_1=t$. We know by Theorem \ref{Thm2} that $\gamma(G)\geq t+2$. To show that $\gamma(G)=t+2$, we simply need to exhibit a dominating set of $G$ of size $t+2$. For convenience, we think of the vertices of $K_{n_i}$ as the elements of $\Z/n_i\Z$. It is straightforward to show that the $t+2$ vertices \[(0,0,0,0\ldots,0),\hspace{.2cm}(1,1,1,1\ldots,1),\hspace{.2cm}\ldots,\hspace{.2cm}(t-1,t-1,t-1,t-1,\ldots,t-1),\] \[(0,1,t,t,\ldots,t),\hspace{.2cm}(1,0,t,t,\ldots,t)\] form a dominating set of $G$. Note that this is the point in the proof where we use the assumption $n_3\geq t+1$.   

To prove the converse, assume $\gamma(G)=t+2$. Theorem \ref{Thm2} shows that $\dfrac{t+1}{2}<n_1$. Since $\gamma(G)=t+2$, we know from the last line in Theorem \ref{Thm1} that $n_1\leq t$. If $n_1=t$, then we are done. Hence, we may assume $n_1\leq t-1$. We simply need to show that $t+1<n_2$. Suppose instead that $n_2\leq t+1$. 

Let $D$ be a dominating set of $G$ with $|D|=\gamma(G)=t+2$. For each $\ell\in\{1,\ldots,t\}$ and each vertex $v$ of $K_{n_\ell}$, let $F_v(\ell)=\{z\in D\colon[z]_\ell=v\}$ as in Lemma \ref{Lem4}. Recall from that lemma that $|F_v(\ell)|\leq 2$ for all $\ell$ and $v$. By the pigeonhole principle, there exists a vertex $u$ of $K_{n_2}$ such that $|F_u(2)|=2$. Since $|D\setminus F_u(2)|=t>n_1$, another application of the pigeonhole principle tells us that there is a vertex $u'$ of $K_{n_1}$ with $|F_{u'}(1)|=2$ and $F_{u'}(1)\subseteq D\setminus F_u(2)$. However, this contradicts Lemma \ref{Lem4} with $\ell=1$, $\ell'=2$, $v=u'$, and $v'=u$.   
\end{proof}

Let $G$ be as in the preceding theorem. We have characterized precisely when $\gamma(G)=t+2$ under the assumption $n_3\geq t+1$. The last two paragraphs of the above proof show that if $\gamma(G)=t+2$ and $n_3\leq t$, then in fact $n_1=n_2=n_3=t$. We leave open the problem of characterizing when $\gamma(G)=t+2$ under the assumption $n_1=n_2=n_3=t$.  


As an application of some of the results from this section, we prove a theorem that will be useful in the next section. 

\begin{thm} \label{cubecorner}
If $G=\prod_{i=1}^4K_{n_i}$, where $2=n_1\leq n_2\leq n_3\leq n_4$, then $\gamma(G)=8$.
\end{thm}

\begin{proof}
As before, we use the elements of $\Z/n_i\Z$ to represent the vertices of $K_{n_i}$. Let $G'=\prod_{i=2}^4 K_{n_i}$. In \cite{valencia}, it is shown that the set
\[D = \{(0,0,0), (0,1,1), (1,0,1), (1,1,0)\}\]
is a dominating set for $G'$. We argue that the set 
\[E= \{(0,0,0,0), (0,0,1,1), (0,1,0,1), (0,1,1,0), (1,0,0,0), (1,0,1,1), (1,1,0,1), (1,1,1,0)\}\]
is a dominating set of $G$. Let $x=(x_1,x_2,x_3,x_4)$ be a vertex of $G$. Since $n_1=2$, $x_1$ is either $0$ or $1$. Because $D$ is a dominating set for $G'$, we know there is some element $(a,b,c)$ of $D$ that is adjacent (in $G'$) to or equal to the triple $(x_2,x_3,x_4)$. Note that $(0,a,b,c),(1,a,b,c)\in E$. We have two cases. If $(x_2,x_3,x_4)$ is adjacent to $(a,b,c)$ in $G'$, then $x$ is adjacent to either $(0,a,b,c)$ or $(1,a,b,c)$ in $G$ depending on the value of $x_1$. If $(x_2,x_3,x_4) = (a,b,c)$, then either $x = (0,a,b,c)$ or $x = (1,a,b,c)$. This proves that $E$ is a dominating set for $G$, so $\gamma(G)\leq 8$. 

To prove that $\gamma(G)\geq 8$, let $s$ be the largest element of $\{1,2,3,4\}$ such that $n_s=2$. We see from Lemma \ref{Lem1} that $\gamma(G)=2^{s-1}\gamma(K_{n_s}\times K_{n_{s+1}}\times\cdots\times K_{n_4})$. Therefore, it suffices to show that $\gamma(K_{n_s}\times K_{n_{s+1}}\times\cdots\times K_{n_4})\geq 2^{4-s}$. If $s=1$, this follows from Theorem \ref{Thm2} with $t=4$. If $s\geq 2$, then this follows from Theorem \ref{Thm1}. 
\end{proof}

\section{Domination in Unitary Cayley Graphs}

Recall that $\gamma(\X) \leq g(n)$, where $\X$ is the unitary Cayley graph of $\Z/n\Z$ and $g$ is Jacobsthal's function. In this section, we provide results about when $\gamma(\X)=g(n)$ and when $\gamma(\X)<g(n)$. First, observe that if $n=p_1^{\alpha_1}p_2^{\alpha_2}\cdots p_t^{\alpha_t}$ is the prime factorization of $n$ and $m=p_1p_2\cdots p_t$ is the radical (also known as the squarefree core) of $n$, then $\gamma(\X)\geq\gamma(X_{\mathbb Z/m\mathbb Z})$ (see Lemma \ref{Lem5}). This observation is useful because it is often convenient to work under the assumption that $n$ is squarefree. Indeed, if $n=p_1p_2\cdots p_t$, where $p_1,p_2,\ldots,p_t$ are distinct primes, then $\X\cong\prod_{i=1}^tK_{p_i}$ and so we can use the theorems from the preceding section to help calculate $\gamma(\X)$.

Let $\omega(n)$ be the number of distinct prime factors of an integer $n>1$. We can explicitly calculate $\gamma(\X)$ when $\omega(n)\leq 3$. If $n=p_1p_2\cdots p_t$ is squarefree (with $p_1<p_2<\cdots<p_t$) and $t=\omega(n)\leq 3$, then it follows from Theorem \ref{Thm1} that 
\begin{equation}\label{Eq7}
\gamma(\X)=\begin{cases} 1, & \mbox{if } \omega(n)=1; \\ 2, & \mbox{if } p_1=\omega(n)=2; \\ 3, & \mbox{if } p_1>\omega(n)=2; \\ 4, & \mbox{if } \omega(n)=3. \end{cases}
\end{equation} Using the following lemma, we will show that $\gamma(\X)=g(n)$ when $n$ is not squarefree and $\omega(n)\leq 3$. 

\begin{lem}\label{Lem6}
Let $n=p_1^{\alpha_1}p_2^{\alpha_2}\cdots p_t^{\alpha_t}$, where $p_1<p_2<\cdots<p_t$ are primes and $\alpha_1,\alpha_2,\ldots,\alpha_t$ are positive integers. If $t\leq 3$ and $\alpha_j\geq 2$ for some $j$, $1\leq j\leq t$, then \[\gamma(\X)\geq\frac{p_1t}{p_1-1}.\]
\end{lem}

\begin{proof} 
Since $\gamma(\X)\geq 2$ (every vertex of $\X$ has degree $\varphi(n)<n-1$, where $\varphi$ is Euler's totient function), the lemma is easy when $t=1$.  The cases $t=2$ and $t=3$ are similar, so we will only prove the lemma in the more difficult case when $t=3$. Let $\gamma=\gamma(\X)$, and assume toward a contradiction that $\gamma<\dfrac{p_1t}{p_1-1}=\dfrac{3p_1}{p_1-1}$. We can rewrite this inequality as $p_1(\gamma-3)<\gamma$, which implies that $\gamma\leq m+2$, where $m=\left\lceil\gamma/p_1\right\rceil$. We know from \eqref{Eq7} and the observation made at the beginning of this section that $\gamma\geq\gamma(X_{\Z/p_1p_2p_3\Z})=4$. Checking some easy cases, we see that this forces $\gamma=m+2$. Let $D=\{d_1,d_2,\ldots,d_\gamma\}$ be a dominating set of $\X$. By the Pigeonhole Principle, there exist $\left\lceil\gamma/p_1\right\rceil=m$ elements of $D$, say $d_1,\ldots,d_m$, that are all congruent to each other modulo $p_1$. 

For the sake of finding a contradiction, assume that $d_{m+1}\equiv d_1\pmod{p_1}$. Let $A$ be the set of vertices $x$ of $\X$ satisfying $x\equiv d_1\pmod{p_1}$ and $x\equiv d_{m+2}\pmod{p_2}$. No vertex in $A$ is adjacent to any element of $D$, so $A\subseteq D$. However, the Chinese remainder theorem tells us that $|A|=p_1^{\alpha_1-1}p_2^{\alpha_2-1}p_3^{\alpha_3}$. Since $p_3\geq 5$ and $\alpha_j\geq 2$ for some $1\leq j\leq 3$, we find that $|A|\geq 10$. This is a contradiction because $|D|=\gamma<\dfrac{3p_1}{p_1-1}\leq 6$. Consequently, $d_{m+1}\not\equiv d_1\pmod{p_1}$. Similarly, $d_{m+2}\not\equiv d_1\pmod{p_1}$. In fact, the exact same argument with the assumption $d_1\equiv d_{m+1}\pmod{p_1}$ replaced by the assumption $d_{m+1}\equiv d_{m+2}\pmod{p_2}$ yields a contradiction, showing that $d_{m+1}\not\equiv d_{m+2}\pmod{p_2}$. 

For vertices $a,b,c$ of $\X$, let $B(a,b,c)$ denote the set of vertices $x$ satisfying $x\equiv a\pmod{p_1}$, $x\equiv b\pmod{p_2}$, and $x\equiv c\pmod{p_3}$. Note that $|B(a,b,c)|=p_1^{\alpha_1-1}p_2^{\alpha_2-1}p_3^{\alpha_3-1}\geq 2$. No element of $B(d_1,d_{m+1},d_{m+2})$ is adjacent to any element of $D$, so $B(d_1,d_{m+1},d_{m+2})\subseteq D$. Since $d_{m+1}\not\equiv d_1\pmod{p_1}$ and $d_{m+2}\not\equiv d_1\pmod{p_1}$, $B(d_1,d_{m+1},d_{m+2})\subseteq \{d_1,d_2,\ldots,d_m\}$. Similarly, $B(d_1,d_{m+2},d_{m+1})\subseteq \{d_1,d_2\ldots,d_m\}$. We saw above that $\gamma<6$, so $m=\left\lceil\gamma/p_1\right\rceil\leq 3$. This implies that \[B(d_1,d_{m+1},d_{m+2})\cap B(d_1,d_{m+2},d_{m+1})\neq\emptyset.\] Say $z\in B(d_1,d_{m+1},d_{m+2})\cap B(d_1,d_{m+2},d_{m+1})$. Then $d_{m+1}\equiv z\equiv d_{m+2}\pmod{p_2}$, which contradicts the last line in the previous paragraph. 
\end{proof}


\begin{thm}\label{Thm4}
If $\omega(n)\leq 3$ and $n$ is not squarefree, then $\gamma(\X)=g(n)$.   
\end{thm} 
\begin{proof}
Let $t=\omega(n)$, and let $p_1$ be the smallest prime factor of $n$. By checking a few simple cases, one may easily show that $g(n)\leq\left\lceil\dfrac{p_1t}{p_1-1}\right\rceil$. The proof now follows from Lemma \ref{Lem6} and the fact that $\gamma(\X)\leq g(n)$. 
\end{proof}

The rest of this section is devoted to studying the set $M := \{n \in \N \ \colon \ \gamma(\X) < g(n)\}$.  

\begin{prop}\label{Prop1}
The set $M = \{n \in \N \ \colon \ \gamma(\X) < g(n)\}$ is infinite. More precisely, the following two infinite sets are contained in $M$:
\begin{enumerate}
\item $\{2p_1p_2 \ \colon \ p_1, p_2\text{ are prime and } 3 \leq p_1< p_2\}$
\item$\{6p_1p_2\ \colon \ p_1, p_2\text{ are prime and } 5 \leq  p_1< p_2 \}$.
\end{enumerate}
\end{prop}

\begin{proof}
We show the harder part---that set $(2)$ is contained in $M$. It is similar to argue that set $(1)$ is contained in $M$. Let $n = 6p_1p_2$ where $5 \leq p_1 <p_2$. By Theorem \ref{cubecorner}, $\gamma(\X) =8$. There is an integer $x$ such that $x\equiv 0\pmod 2$, $x\equiv -1\pmod 3$, $x\equiv -3\pmod {p_1}$, and $x\equiv -5\pmod {p_2}$. None of the integers $x+i$ for $0\leq i\leq 8$ are relatively prime to $n$, so $g(n)\geq 10$.  
\end{proof}

The previous result tells us that for infinitely many $n$, the domination number of $\X$ is strictly less than the Jacobsthal function evaluated at $n$. One might suspect that we have only been able to obtain this result by restricting our attention to integers $n$ with a small number of prime factors. Could it be true that $\gamma(\X)=g(n)$ whenever $\omega(n)$ is sufficiently large? The next theorem answers this question in the negative. 

In fact, we can prove something stronger. Recall from the introduction that a total dominating set of a graph $G$ is a set $D\subseteq V(G)$ such that every vertex of $G$ is adjacent to an element of $D$. The total domination number of $G$, denoted $\gamma_t(G)$, satisfies the easy inequality $\gamma_t(G)\geq\gamma(G)$. Using Theorem \ref{Thm4}, one can show that $\gamma_t(\X)=g(n)$ whenever $\omega(n)\leq 3$. Therefore, it is natural to ask if $\gamma_t(n)=g(n)$ in general. It turns out that this is not the case. Let $M_t=\{n\in\N\colon\gamma_t(\X)<g(n)\}$. Note that $M_t$ is a subset of the set $M$ from Proposition \ref{Prop1}. Let $\omega (M_t) = \{\omega(m) \ \colon \ m \in M_t \}$ where, again, $\omega(j)$ denotes the number of distinct prime factors of $j$. 

\begin{thm}\label{Thm6}
The set $\omega(M_t)$ is unbounded.
\end{thm}

\begin{proof}
Let $j \in \N$. We construct $n\in M_t$ such that $\omega (n)\geq j$. Choose a prime $q$ such that $q \equiv 1 \pmod 3$, and $\frac{2(q-1)}{3}+2 \geq j$. Let $k =  \frac{2(q-1)}{3}$, and fix $k$ primes $p_1, \dots ,p_k$ with each $p_i \geq q+3$. Let $n = 3qp_1 \cdots p_k$. For $0\leq i\leq q+2$ with $i\equiv 0\pmod 3$, let $a_i=3$. Let $a_1=a_{q+1}=q$. If $s_\ell$ denotes the $\ell^\text{th}$ smallest element of $\{2,3,\ldots,q\}$ that is not a multiple of $3$, then let $a_{s_\ell}=p_\ell$. This defines $a_i$ for all $0\leq i\leq q+2$. We know by the Chinese remainder theorem that there is an integer $z$ satisfying $z\equiv -i\pmod{a_i}$ for all $i$. The set $\{z+i\colon 0\leq i\leq q+2\}$ consists of $q+3$ consecutive integers, none of which are relatively prime to $n$. Thus, $g(n)\geq q+4$. 

Choose a vertex $y$ of $\X$ such that $y\equiv1\pmod 3$, $y\equiv-1\pmod q$, and $y\equiv-1\pmod{p_i}$ for all $1\leq i\leq k$. Let $D = \{0,1, \dots , q+1,y\}$. We show that $D$ is a total dominating set of $\X$. Because $|D|=q+3<g(n)$, this will prove that $n\in M_t$. 

Suppose a vertex $x$ is not adjacent to any element of $D\setminus\{y\}$. We will show that $x$ is adjacent to $y$. The set $S = \{ x, x-1, x-2, \dots , x-(q+1) \}$ consists of $q+2$ consecutive integers, none of which are coprime to $n$. For $r\in\N$, let $B(r)=\{s\in S\colon s\equiv 0\pmod r\}$. Observe that $|B(3)|\leq\left\lceil\frac{q+2}{3}\right\rceil=\frac{q+2}{3}$, $|B(q)|\leq \left\lceil\frac{q+2}{q}\right\rceil=2$, and $|B(p_i)|\leq \left\lceil\frac{q+2}{p_i}\right\rceil=1$ for all $1\leq i\leq k$. Since no element of $S$ is coprime to $n$, \[S=B(3)\cup B(q)\cup\left(\bigcup_{i=1}^kB(p_i)\right).\] Therefore, \[q+2=\left|B(3)\cup B(q)\cup\left(\bigcup_{i=1}^kB(p_i)\right)\right|\leq|B(3)|+|B(q)|+\sum_{i=1}^k|B(p_i)|\leq\frac{q+2}{3}+2+k=q+2.\] The inequalities in the previous line must actually be equalities, and this implies that the sets $B(3),B(q),B(p_1),B(p_2),$ $\ldots,B(p_k)$ are disjoint. Since $|B(q)|=2$ and $|S|=q+2$, either $x\in B(q)$ or $x-1\in B(q)$. In particular, $x\not\equiv -1\equiv y\pmod q$. For each $1\leq i\leq k$, it follows from the fact that $|B(p_i)|=1$ and the assumption that $p_i\geq q+3$ that $x\not\equiv -1\equiv y\pmod{p_i}$.  

Suppose $x\equiv 1\pmod 3$. Then $x-1\in B(3)$. Since $B(3)$ and $B(q)$ are disjoint, $x-1\not\in B(q)$. We noted above that either $x\in B(q)$ or $x-1\in B(q)$, so we must have $x\in B(q)$. This implies that $x-q\in B(q)$. However, $x-q\in B(3)$ since $q\equiv 1\pmod 3$. This is a contradiction. We conclude that $x\not\equiv 1\equiv y\pmod 3$, so $x$ is adjacent to $y$ as desired.  
\end{proof}

In light of the preceding theorem, it would be interesting to know if $g(n)-\gamma(\X)$ can be arbitrarily large. We currently have no evidence to either support or refute the claim  that this is the case. 

\section{Upper domination in direct products of multipartite graphs}

Recall that the upper domination number $\Gamma(G)$ of a graph $G$ is the maximum size of a minimal dominating set of $G$. 
We will make use of a classical result due to Ore \cite{ore}, which states that a dominating set $D$ of a graph $G$ is minimal if and only if for each $d \in D$, one of the following conditions holds:
\begin{enumerate}
\item $d$ is not adjacent to any vertex of $D$
\item there exists a vertex $p$ in $V(G) \backslash D$ such that $d$ is the only neighbor of $p$ in $D$.
\end{enumerate}

Let $D$ be a minimal dominating set of a graph $G$. We say a vertex $d\in D$ is \emph{lonely} if it is not adjacent to any vertex of $D$. Otherwise, we say $d$ is \emph{social}. In other words, a vertex $d\in D$ is social if it is adjacent to an element of $D$. If $d$ is a social vertex of $D$ and $p \in V(G) \backslash D$ is such that $d$ is the only neighbor of $p$ in $D$, then we call $p$ a \emph{private neighbor} of $d$. The next proposition forms the motivation for the following conjecture, which is the focus of this section. 

\begin{prop}\label{Prop2}
If $G=\displaystyle \prod_{i=1}^t K[a_i,b_i]$ with $2\leq b_1\leq b_2\leq\cdots\leq b_t$, then $\displaystyle\Gamma(G) \geq \frac{1}{b_1}\prod_{i=1}^t a_ib_i$.
\end{prop}

\begin{proof}
Let $P$ be one of the partite sets of $K[a_1,b_1]$. Let $D$ be the set of all vertices of $G$ whose first coordinate is an element of $P$. One can check that $D$ is a minimal dominating set of $G$ and that $|D| = \frac{1}{b_1}\prod_{i=1}^t a_ib_i$.
\end{proof} 

\begin{conj}\label{upperdom}
If $G=\displaystyle \prod_{i=1}^t K[a_i,b_i]$ with $2\leq b_1\leq b_2\leq\cdots\leq b_t$, then $\displaystyle\Gamma(G)=\frac{1}{b_1}\prod_{i=1}^t a_ib_i$.
\end{conj}

In the context of unitary Cayley graphs, this conjecture states that if $p$ is the smallest prime factor of $n$, then $\Gamma(\X)= n/p$.

We provide some partial results supporting this conjecture.

\begin{lem}\label{Lem7}
Let $G$ be as in Conjecture \ref{upperdom}. The vertices of $G$ can be partitioned into $\frac{1}{b_1}\prod_{i=1}^ta_ib_i$ cliques of size $b_1$. 
\end{lem}

\begin{proof}
We use the elements of $\Z/a_ib_i\Z$ to represent the vertices of $K[a_i,b_i]$; two vertices $x$ and $y$ are adjacent if and only if $x\not\equiv y\pmod{b_i}$. The proof is by induction on $t$. If $t=1$, then the $a_1$ sets of the form $\{mb_1,mb_1+1,\ldots,mb_1+(b_1-1)\}$ for $0\leq m\leq a_1-1$ are disjoint cliques of $G$. Now, suppose $t\geq 2$, and let $P=\prod_{i=1}^{t-1}a_ib_i$. Assume inductively that the vertices of $\prod_{i=1}^{t-1}K[a_i,b_i]$ can be partitioned into $P/b_1$ cliques $\mathcal C_1,\ldots,\mathcal C_{P/b_1}$, each of size $b_1$. Let $\mathcal C_i=\{(c_{i,j,1},c_{i,j,2},\ldots,c_{i,j,t-1})\colon 1\leq j\leq b_1\}$. For $1\leq\ell\leq a_tb_t$, let $\mathcal C_i^{(\ell)}=\{(c_{i,j,1},c_{i,j,2},\ldots,c_{i,j,t-1},\ell+j)\colon 1\leq j\leq b_1\}$, where the coordinates $\ell+j$ are taken modulo $a_tb_t$. The sets $\mathcal C_i^{(\ell)}$ for $1\leq i\leq P/b_1$ and $1\leq\ell\leq a_tb_t$ are disjoint cliques of $G$.   
\end{proof}

\begin{thm} \label{evenupdom}
Let $G$ be as in Conjecture \ref{upperdom}. If $D$ is a minimal dominating set of $G$ with $\ell$ lonely vertices and $s$ social vertices, then $b_1\ell+2s\leq \prod_{i=1}^ta_ib_i$. In particular, Conjecture \ref{upperdom} holds if $b_1=2$. 
\end{thm}

\begin{proof}
Let $n=\prod_{i=1}^ta_ib_i$. According to the preceding lemma, we can partition the vertices of $G$ into $n/b_1$ cliques $\mathcal C_1,\ldots,\mathcal C_{n/b_1}$, each of size $b_1$. For each lonely vertex $u\in D$, let $B(u)$ be the unique clique from the list $\mathcal C_1,\ldots,\mathcal C_{n/b_1}$ that contains $u$. For each social vertex $v\in D$, choose a private neighbor $p_v$ of $v$ (that is, $p_v$ is adjacent to $v$ but is not adjacent to any other element of $D$), and let $B(v)=\{v,p_v\}$. Using the definitions of lonely vertices, social vertices, and private neighbors, the reader may verify that the sets $B(d)$ for $d\in D$ are disjoint. As a consequence, \[b_1\ell+2s=\left|\bigcup_{d\in D}B(d)\right|\leq |V(G)|=n.\] If $b_1=2$, then we find that \[|D|=\ell+s\leq n/2.\] Combined with Proposition \ref{Prop2}, this proves Conjecture \ref{upperdom} in the case $b_1=2$.       
\end{proof}

Note that the last line in Theorem \ref{evenupdom} implies that $\Gamma(\X) = n/2$ if $n$ is even. The following proposition and theorem prove Conjecture \ref{upperdom} in some additional cases.

\begin{prop}\label{Prop3}
Conjecture \ref{upperdom} is true if $t\leq 2$.
\end{prop}

\begin{proof}
Let $G$ be as in Conjecture \ref{upperdom}, and let $n=\prod_{i=1}^t a_ib_i$. If $n=b_1$, then $G=K_{b_1}$ is a complete graph with $\Gamma(G)=1=n/b_1$ as desired. Thus, we may assume $n>b_1$. Suppose $D$ is a dominating set of $G$ and $|D|>n/b_1$. Note that $n/b_1$ is an integer that is greater than $1$, so $|D|\geq 3$. Lemma \ref{Lem7} tells us that $V(G)$ can be partitioned into $n/b_1$ cliques. Since $|D|>n/b_1$, there exist adjacent vertices $d,d'\in D$. If $t=1$ (so $G=K[a_1,b_1]$), then $d$ and $d'$ are in different partite sets. This means that $\{d,d'\}$ is a dominating set of $G$, so $D$ cannot be a minimal dominating set. This proves the conjecture in the case $t=1$. 

Next, assume $t=2$. Write $d=(x,y)$ and $d'=(x',y')$, where $x\neq x'$ and $y\neq y'$. The only vertices of $G$ not dominated by $\{d,d'\}$ are $(x,y')$ and $(x',y)$. Let $d''$ and $d'''$ be elements of $D$ that dominate $(x,y')$ and $(x',y)$, respectively. We find that $\{d,d',d'',d'''\}$ is a dominating set of $G$, so $D$ is not minimal unless $|D|\leq 4$. Because $|D|>n/b_1$, this proves Conjecture \ref{upperdom} when $t=2$ and $n/b_1\geq 4$. Thus, we may assume $t=1$ and $n/b_1\leq 3$. Theorem \ref{evenupdom} tells us that the conjecture is true when $b_1=2$, so we may also assume $b_1\geq 3$. Since $n/b_1=a_1a_2b_2$, this forces $a_1=a_2=1$ and $b_1=b_2=3$. In other words, $G=K_3\times K_3$. We leave the reader to check that $\Gamma(K_3\times K_3)=3$ so that Conjecture \ref{upperdom} is true in this final case.      
\end{proof}

\begin{thm}\label{Thm5}
Let $G$ be as in Conjecture \ref{upperdom}, and assume $t\geq 3$. Let $n=\prod_{i=1}^ta_ib_i$. Let $\kappa_1,\ldots,\kappa_t$ be an enumeration of $\{1,\ldots,t\}$ such that $a_{\kappa_1}b_{\kappa_1}\leq a_{\kappa_2}b_{\kappa_2}\leq\cdots\leq a_{\kappa_t}b_{\kappa_t}$. Conjecture \ref{upperdom} is true if \[t(t-1)(t-2)+3\leq a_{\kappa_1}a_{\kappa_2}b_{\kappa_2}a_{\kappa_3}b_{\kappa_3}.\]  
\end{thm}

\begin{proof}
For the sake of finding a contradiction, assume that 
\begin{equation}\label{Eq8}
t(t-1)(t-2)+3\leq a_{\kappa_1}a_{\kappa_2}b_{\kappa_2}a_{\kappa_3}b_{\kappa_3} 
\end{equation} and that $D$ is a minimal dominating set of $G$ with $|D|>n/b_1$. Write $D=L\cup S$, where $L$ is the set of lonely vertices in $D$ and $S$ is the set of social vertices. For each $v\in S$, choose a private neighbor $p_v$ of $v$. Let $P=\{p_v\colon v\in S\}$. For convenience, we also put $p_v=v$ for each $v\in L$. By Lemma \ref{Lem7}, we can partition $V(G)$ into cliques $\mathcal C_1,\ldots,\mathcal C_{n/b_1}$, each of size $b_1$. For each vertex $u$, let $C(u)$ denote the unique clique from the set $\{\mathcal C_1,\ldots,\mathcal C_{n/b_1}\}$ that contains $u$.  

We claim that either $S$ or $P$ contains a clique of size $3$. To see this, suppose $S$ does not contain a clique of size $3$. Let $\mathcal A_k=\{\mathcal C_i\colon 1\leq i\leq n/b_1, |D\cap\mathcal C_i|=k\}$ be the family of cliques from the set $\{\mathcal C_1,\ldots,\mathcal C_{n/b_1}\}$ that contain exactly $k$ elements of $D$. Let $\mathcal B_k=\{v\in D\colon C(v)\in\mathcal A_k\}$.  Note that $L\subseteq\mathcal B_1$. The assumption that $S$ contains no cliques of size $3$ implies that $\mathcal A_k=\emptyset$ (hence, $\mathcal B_k=\emptyset$) when $k\geq 3$. We find that \[|\mathcal A_0|+|\mathcal A_1|+|\mathcal A_2|=n/b_1<|D|=|\mathcal A_1|+2|\mathcal A_2|,\] so $2|\mathcal A_0|<2|\mathcal A_2|=|\mathcal B_2|$. It is straightforward to check that if $v\in \mathcal B_2$, then $C(p_v)\in\mathcal A_0$. By the pigeonhole principle and the fact that $2|\mathcal A_0|<|\mathcal B_2|$, there exist distinct $v,v',v''\in\mathcal B_2$ such that $C(p_v)=C(p_{v'})=C(p_{v''})$. The vertices $p_v,p_{v'},p_{v''}$ form a clique of size $3$    contained in $P$.   

We now consider two cases. 

{\bf Case 1.} $S$ contains a clique $\{x_1,x_2,x_3\}$ of size $3$. \\
Let $U$ be the set of vertices of $G$ that are not dominated by any element of $\{x_1,x_2,x_3\}$. For each $v\in D\setminus\{x_1,x_2,x_3\}$, $p_v\in U$. Consequently, $|D|\leq 3+|U|$. For each $u\in U$, there are distinct indices $i_1,i_2,i_3\in\{1,\ldots,t\}$ such that the $i_j^{\text{th}}$ coordinate of $u$ is the same as the $i_j^{\text{th}}$ coordinate of $x_j$. Using \eqref{Eq8} and the fact that $x_1,x_2,x_3$ are pairwise adjacent, we find that \[|D|\leq 3+|U|\leq 3+\sum_{i_1,i_2,i_3}\frac{n}{a_{i_1}b_{i_1}a_{i_2}b_{i_2}a_{i_3}b_{i_3}}\leq 3+t(t-1)(t-2)\frac{n}{a_{\kappa_1}b_{\kappa_1}a_{\kappa_2}b_{\kappa_2}a_{\kappa_3}b_{\kappa_3}}\] \[\leq 3+\frac{n}{b_{\kappa_1}}-3\frac{n}{a_{\kappa_1}b_{\kappa_1}a_{\kappa_2}b_{\kappa_2}a_{\kappa_3}b_{\kappa_3}}\leq \frac{n}{b_{\kappa_1}}\leq\frac{n}{b_1}<|D|.\] This is a contradiction.  

{\bf Case 2.} $P$ contains a clique $\{p_{x_1},p_{x_2},p_{x_3}\}$ of size $3$. \\ 
In this case, let $U$ be the set of vertices of $G$ that are not dominated by any element of $\{p_{x_1},p_{x_2},p_{x_3}\}$. Note that $|D|\leq 3+|U|$ because $D\subseteq U\cup\{x_1,x_2,x_3\}$. For each $u\in U$, there are distinct indices $i_1,i_2,i_3\in\{1,\ldots,t\}$ such that the $i_j^{\text{th}}$ coordinate of $u$ is the same as the $i_j^{\text{th}}$ coordinate of $p_{x_j}$. As in Case 1, we can use \eqref{Eq8} and the fact that $p_{x_1},p_{x_2},p_{x_3}$ are pairwise adjacent to derive the contradiction \[|D|=3+|U|\leq 3+\sum_{i_1,i_2,i_3}\frac{n}{a_{i_1}b_{i_1}a_{i_2}b_{i_2}a_{i_3}b_{i_3}}\leq 3+t(t-1)(t-2)\frac{n}{a_{\kappa_1}b_{\kappa_1}a_{\kappa_2}b_{\kappa_2}a_{\kappa_3}b_{\kappa_3}}<|D|. \qedhere\]    
\end{proof}

\begin{cor}\label{Cor1}
Conjecture \ref{upperdom} is true if $t=3$. 
\end{cor}

\begin{proof}
Let $G=\prod_{i=1}^3K[a_i,b_i]$. Let $\kappa_1,\kappa_2,\kappa_3$ be an enumeration of $\{1,2,3\}$ such that $a_{\kappa_1}b_{\kappa_1}\leq a_{\kappa_2}b_{\kappa_2}\leq a_{\kappa_3}b_{\kappa_3}$, and let $m=a_{\kappa_1}a_{\kappa_2}b_{\kappa_2}a_{\kappa_3}b_{\kappa_3}$. Theorem \ref{Thm5} tells us that Conjecture \ref{upperdom}  is true if $9\leq m$. Hence, we may assume $m\leq 8$. This implies that $a_{\kappa_2}b_{\kappa_2}=2$, so $a_{\kappa_1}=a_{\kappa_2}=1$ and $b_{\kappa_1}=b_{\kappa_2}=2$. It follows that $b_1=2$ and $G=K_2\times K_2\times K[a_{\kappa_3},b_{\kappa_3}]$. If we let $H=K_2\times K[a_{\kappa_3},b_{\kappa_3}]$, then we know from \eqref{Eq1} that $G\cong H\oplus H$. This tells us that $\Gamma(G)=2\Gamma(H)$. We know from Proposition \ref{Prop3} that $\Gamma(H)=a_{\kappa_3}b_{\kappa_3}$, so \[\Gamma(G)=2a_{\kappa_3}b_{\kappa_3}=\frac{1}{2}\prod_{i=1}^3a_ib_i=\frac{1}{b_1}\prod_{i=1}^3a_ib_i. \qedhere\]  
\end{proof}

\section{Conjectures and Open Problems}

Many of the proofs of the results in Section 2 rely on Lemma \ref{Lem2}. As mentioned in Section 2, it would be useful to have stronger versions of this lemma. Also, recall the problem that we mentioned immediately after the proof of Theorem \ref{Thm3}. Specifically, with $G$ as in that theorem, we would like to have a characterization of when $\gamma(G)=t+2$ under the additional assumption $n_1=n_2=n_3=t$. 
 
In Theorem \ref{Thm6} we showed that there exist integers $n$ with arbitrarily many distinct prime factors such that $\gamma(\X)\leq\gamma_t(\X)<g(n)$. As mentioned at the end of Section 3, it is not known if $g(n)-\gamma(\X)$ can be arbitrarily large. We pose the problem of determining whether there are integers $n$ with $\omega(n)$ arbitrarily large and $\gamma(\X)\leq g(n)-2$. In fact, it remains open to find a single integer $n$ such that $\gamma_t(\X)\leq g(n)-2$.

Recall that we have proven Conjecture \ref{upperdom} in some cases. In particular, we have shown that the 
conjecture is true when $t\leq 3$. However, the full conjecture is still open. 
One particularly attractive special case of the conjecture that remains open is  
that in which $G=\prod_{i=1}^tK_3$. 
It would also be interesting  
to prove the slightly weaker form of Conjecture \ref{upperdom} stating that $\Gamma(\X)=n/p$, where $p$ is the smallest prime factor of $n$.

\section{Acknowledgments}

We would like to thank Joe Gallian for providing extraordinary support and encouragement as well as for reading through this paper at the 2017 REU at the University of Minnesota Duluth. The REU is supported by grant NSF / DMS-1659047 and provided an amazing working environment. Additional funding for Sumun Iyer came from the Clare Boothe Luce Program of the Henry Luce Foundation. 

The authors also thank Evan O'Dorney, David Rolnick, and the anonymous referees for reading through this article and providing helpful commentary.

\end{document}